\documentclass[11pt]{amsart}
\usepackage{amsmath, amsthm, amscd, amsfonts,}

 \newtheorem{theorem}{Theorem}[section]
\newtheorem{proposition}[theorem]{Proposition}

\newtheorem{corollary}[theorem]{Corollary}

\newtheorem{lemma}[theorem]{Lemma}
\theoremstyle{definition}
\newtheorem{definition}[theorem]{Definition}
\newtheorem{example}[theorem]{Example}
\newtheorem{counterexample}[theorem]{Counterexample}

\begin{document}

\title[When a quotient of a distributive lattice is a boolean algebra]
{when a quotient of a distributive lattice is a boolean algebra}

 \author[H.Barzegar ]{H.Barzegar }

 \address{Department of Mathematics, Tafresh University, Tafresh 39518-79611, Iran.}

 \email{h56bar@tafreshu.ac.ir}


 \subjclass[2010]{03G10, 06Dx, 03G05}

 \keywords{Distributive lattice, Boolean algebra, Congruence, Ideal, Filter }

\begin{abstract}
In this article, we introduce a lattice congruence with respect to a nonempty ideal $I$ of a distributive lattice $L$ and a derivation $d$ on $L$ denoted by $\theta_I^d$.
 We investigate some necessary and sufficient conditions for the
quotient algebra $L/\theta_I^d$ to become a Boolean algebra.
\end{abstract}

\maketitle
\section{Introduction and Preliminaries}
In Calculus a derivation is a linear operation $d$ with an
additional property $d(f.g)=d(f)g+f.d(g)$. Based on this property,
several authors generalize derivation in some other branches of
mathematics. First, the notion of derivation had studied in rings
and near-rings \cite{Bell, Posner}. After that some authors
applied the notion of derivation in other branches of mathematics,
for example,  Jun and Xin \cite{Jun} in  BCI-algebras and
\cite{Ferrari, alshehri, Xin1, Xin2, Szasz} in lattices. In
\cite{Xin2}, Xin et al, for lattices with a greatest element,
modular lattices, and distributive lattices  gave some equivalent
conditions, under which a derivation is isotone. They
characterized modular lattices and distributive lattices by
isotone derivations. Also Xin  answered to some other questions
about the relations among derivations, ideals, and fixed sets  in
\cite{Xin1}.

 Lattices and Boolean algebras play a significant  structural role in computer
science and logic as well.  Recall that a Boolean algebra is a bounded, complemented
distributive lattice. So Boolean algebras have
a very closed relation to lattices.  One of the common subject in all kinds of algebras are congruences.
The study of congruence relations on lattices and an inter-relation between
ideals and congruences in a lattice had became
a special interest to many authors, for example \cite{adam, gratzer1, luo}.

The main result of this manuscript is to obtain a necessary and sufficient condition in which the quotient lattice $L/\theta$ is a Boolean algebra. Let us explain and motivate what we intend to do in this article. In \cite{sambasiva}, two types of congruences are introduced in a distributive
lattice, both are defined in terms of derivations.  We found some mistakes in that paper and  our attempt are led to the paper \cite{Erratum}. After that, we are interested to generalize the work
 on a distributive lattice with a nonempty ideal. To this end, we first recall some ingredients needed in the sequel very briefly in this section. For more information see, for example, \cite{birkhoff, Ferrari, Xin2}.\\
 Throughout the paper $L$ stands for a distributive lattice.
A least element, so called the bottom element, of a distributive
lattice, if exists, is denoted by $\bot_L({\rm or} \bot)$ and a
greatest one, so called the top element, is denoted by $\top_L({\rm
or} \top)$. In which case $L$ is called a bounded lattice.
By a lattice map (or homomorphism), we mean a map $f:A\rightarrow B$ between two lattices which preserves binary operations $\vee$ and $\wedge$.
Recall that a non-empty subset $I$ of $L$ is called an {\it ideal (filter)} of $L$ if $a\vee b \in A\ (a \wedge b \in A)$ and $a \wedge x \in A\ (a\vee x \in A)$ whenever $a, b \in A$ and $x \in L$. An equivalence relation $\theta$ defined on $L$ is said to be a lattice congruence on $L$
if it satisfies the following conditions, $a\theta b$ implies $(a\vee c)\theta (b\vee c)$ and $(a\wedge c)\theta (b\wedge c)$, for all $a,b,c\in L$.
 \begin{definition}\label{def1}\cite{Ferrari}
For a distributive lattice $L$, a function $d : L\rightarrow L$ is called a derivation
on $L$, if for all $x, y \in L$:

 ${\rm (i)}$ $d(x\wedge y) = (d(x)\wedge y) \vee (x\wedge d(y))$.

 ${\rm (ii)}$ $d(x\vee y) = d(x)\vee d(y)$.
\end{definition}

 In \cite[Th. 3.21]{Xin2} was shown that the condition (i) can be simplified in the following way which we use it through the paper from now on.
\begin{lemma}\label{derivation}\cite{Xin2}
If $L$ is a distributive lattice, then $d:L\rightarrow L$ is a derivation if and only if the following conditions hold:

 ${\rm (i)}$ $d(x\wedge y) = d(x)\wedge y = x\wedge d(y)$.

${\rm (ii)}$ $d(x\vee y) = d(x)\vee d(y)$.
\end{lemma}
One can find the proof of the following lemma in \cite{Ferrari} and \cite{Xin1} which we also need to proceed.

 \begin{lemma}\label{property-derivation}
Let $d:L\rightarrow L$ be a derivation and $x,y\in L$.

 ${\rm (i)}$ If $L$ has a bottom element $\bot$, then $d(\bot)=\bot$.

 ${\rm (ii)}$ $d(x)\leq x$.

 ${\rm (iii)}$ $d(d(x))=d(x)$.

 ${\rm (iv)}$ If $x\leq y$, then $d(x)\leq d(y)$.

 ${\rm (v)}$ If $I$ is an ideal of $L$, then $d(I)\subseteq I$.

 ${\rm (vi)}$ If $L$ has a top element $\top$, then $d(x)=x\wedge d(\top)$.

 ${\rm (vii)}$ Let $L$ have a top element $\top$. If $x\leq d(\top)$, then $d(x)=x$ and if $x\geq d(\top)$, then $d(x)=d(\top)$.
\end{lemma}

As a consequence of the part (iii) of Lemma \ref{property-derivation}, we will have the following corollary.
\begin{corollary}\label{der-hom}
Every derivation $d:L\rightarrow L$ is a lattice homomorphism.
\end{corollary}

In sequence, in Section 2 we extend the concepts of \cite{sambasiva} of the distributive lattice with zero, $0$, (bottom element) to a distributive lattice with a nonempty ideal $I$ instead of $0$.
 Section 3 is devoted to the case where a distributive lattice $L$ is an atomic or more general is an $I$-atomic lattice. Our main results, the goal of this article, are become in section 4. There, we will show the best derivation on $L$ such that $L/\theta_I^d$ become a Boolean algebra is an identity. Finally we demonstrate some necessary and sufficient conditions under which  $L/\theta_I^d$ to be a Boolean algebra.\\

 \section{Congruences and ideals in a distributive lattice with respect to a derivation}
In this section we generalize the article \cite{sambasiva}, from a
distributive lattice with zero, $0$, (bottom element) to distributive
lattice and use a nonempty ideal $I$ instead of $0$. In what
follows we introduce some especial ideals and congruences with
respect to a nonempty ideal and a derivation on distributive
lattices. After that we study some essential properties of this
congruence, purposefully to use in Sections 3 and 4. Note that
most of the definitions of this part have been selected from the  reference
\cite{sambasiva}.

 Suppose $L$  is a distributive lattice, $I$ a nonempty ideal of $L$, $a\in L$ and $d$ a derivation on $L$.  By definition, we consider $ker_{_I}d=d^{-1}(I)=\{x\in L\mid d(x)\in I\}$ and $(a)_I^d=\{x\in L\mid a\wedge x\in ker_{_I}d\}=\{x\in L\mid d(a\wedge x)\in I\}$.

Another view of the subset $(a)_I^d$ of $L$ is of the form $(a)_I^d=(d\circ\lambda_a)^{-1}(I)$, in which $\lambda_a:L\rightarrow L$ is a derivation defined by
$\lambda_a(x)=a\wedge x$.

All parts of the following lemma will be used in the next results of the manuscript.
\begin{lemma}\label{property}
For each $a,b\in L$,

 ${\rm (i)}$~ $ker_{_I}d$ and $(a)_I^d$ are ideals of $L$.

${\rm (ii)}$ if $a\leq b$, then $(b)_I^d\subseteq (a)_I^d$.

${\rm (iii)}$~ $(a\vee b)_I^d=(a)_I^d\cap (b)_I^d$.

${\rm (iv)}$~ $I\subseteq ker_{_I}d\subseteq (a)_I^d$.

${\rm (v)}$ $a\in ker_{_I}d$ {\rm iff} $a\in (a)_I^d$ {\rm iff} $(a)_I^d=L$.

${\rm (vi)}$~ $\bigcap_{_{a\in L}} (a)_I^d=ker_{_I}d$.

 ${\rm (vii)}$~ $a\in (b)_I^d$ if and only if $b\in (a)_I^d$.

${\rm (ix)}$~ let $(a)_I^d\neq L$, then $\bigcap_{_{b\in (a)_I^d}} (b)_I^d\neq ker_{_I}d$.

 ${\rm (x)}$~ if $I$ and $J$ are ideals of $L$ in which $I\subseteq J$, then $ker_Id\subseteq ker_Jd$ and $(a)_I^d\subseteq (a)_J^d$, for each $a\in L$.
\end{lemma}
\begin{proof}
We shall prove only unclear statements.

${\rm (i)}$~ We show only $(a)_I^d$ is an ideal of $L$. Let $x,y\in (a)_I^d$. Then $a\wedge d(x), a\wedge d(y)\in I$. Thus $a\wedge d(x\vee y)=(a\wedge d(x))\vee (a\wedge d(y))\in I$ and hence $x\vee y\in (a)_I^d$. Now let $x\in (a)_I^d$ and $y\in L$. Then $a\wedge d(x)\in I$. Thus $a\wedge d(x\wedge y)=(a\wedge d(x))\wedge d(y)\in I$ which implies $x\wedge y\in (a)_I^d$.



 ${\rm (v)}$~ If $a\in ker_Id$, by (iv), $a\in (a)_I^d$. Now consider $x\in L$ and $a\in (a)    _I^d$. Then $d(a)=d(a\wedge a)\in I$. So $x\wedge d(a)\in I$, which implies $x\in (a)_I^d$. Thus $(a)_I^d=L$. For the converse, let $(a)_I^d=L$. Then $a\in (a)_I^d$, which implies $d(a)\in I$. Therefore $a\in ker_Id$.


${\rm (ix)}$~ By (vii), $a\in \bigcap_{_{b\in (a)_I^d}} (b)_I^d$,
and applying (v), deduces that $\bigcap_{_{b\in (a)_I^d}}
(b)_I^d\neq ker_{_I}d$.
\end{proof}

 Now we introduce a binary relation on a distributive lattice with respect to an ideal and a derivation. The following proposition, which has an easy proof, shows that this binary relation is a lattice congruence.
\begin{proposition}
For an ideal $I$ of $L$, a binary relation $\theta_I^d$ defined as follow is a lattice congruence.
 $$x\theta_I^d y~~ {\rm iff}~~ (x)_I^d=(y)_I^d$$
\end{proposition}

An element $a\in L$ is called a {\it kernel element} with respect to an ideal $I$, if $(a)_I^d=ker_{_I}d$. Let us denote the set of all kernel elements with respect to the ideal $I$ of $L$ by ${\mathcal K}_I^d$.

 If $I=L$, then $I=ker_Id=(a)_I^d={\mathcal K}_I^d=L$ and hence $\theta_I^d=\nabla=\{(a,b)\mid a,b\in L\}$, which implies that $L/\theta_I^d$ is a singleton set. So, from now on, each ideal considers to be nontrivial $(I\neq L)$.

 \begin{lemma}\label{property filter}
${\rm (i)}$~ The subset ${\mathcal K}_I^d$ of $L$ is a filter, whenever ${\mathcal K}_I^d\neq \emptyset$.

${\rm (ii)}$~ $ker_Id=L$ if and only if ${\mathcal K}_I^d=L$.

 ${\rm (iii)}$~ If $(a)_I^d$ and ${\mathcal K}_I^d$ are nontrivial, then ${\mathcal K}_I^d\cap (a)_I^d=\emptyset$.

 ${\rm (iv)}$~ $(x)_I^d=(d(x))_I^d$ and $x\theta_I^d d(x)$, for all $x\in L$.

 ${\rm (v)}$~ If $x\theta_I^d y$, then $d(x)\theta_I^d d(y)$.

 \end{lemma}
\begin{proof}
${\rm (i)}$~ Let $a,b\in {\mathcal K}_I^d$ and $c\in L$. By Lemma \ref{property}(iv), $ker_Id\subseteq (a\wedge b)_I^d$. For the converse, let $x\in (a\wedge b)_I^d$. Then $a\wedge d(b\wedge x)=d((a\wedge b)\wedge x)\in I$ and hence $b\wedge x\in (a)_I^d= ker_Id$. So $b\wedge d(x)=d(b\wedge x)\in I$, which implies $x\in (b)_I^d=ker_Id$. Thus $a\wedge b\in {\mathcal K}_I^d$. Also $a\vee c\in {\mathcal K}_I^d$, by Lemma \ref{property}(iii) and \ref{property}(iv).

 To prove (ii), apply Lemma \ref{property}(v) and for (iii), apply Lemma \ref{property}(vii) and \ref{property}(v).

 ${\rm (iv)}$~ By Lemma \ref{property-derivation}(ii), $d(x)\leq x$ and hence $(x)_I^d\subseteq (d(x))_I^d$. Let $y\in (d(x))_I^d$. Hence $d(y\wedge x)=d(y\wedge d(x))\in I$, which implies $y\in (x)_I^d$. Thus $(x)_I^d= (d(x))_I^d$.
\end{proof}

 The following proposition shows that the quotient lattice $L/\theta_I^d$ is a bounded lattice.
\begin{proposition}\label{bounded}
For a nontrivial ideal $I$ of $L$, the distributive lattice $L/\theta_I^d$ is a bounded lattice with

 ${\rm (i)}$ $\bot_{L/\theta_I^d}=ker_I d$,

 ${\rm (ii)}$ $\top_{L/\theta_I^d}={\mathcal K}_I^d$ whenever ${\mathcal K}_I^d\neq \emptyset$.
\end{proposition}
\begin{proof}
${\rm (i)}$ Let $a\in ker_Id$. By Lemma \ref{property}, for each $b\in ker_Id$, $(a)_I^d=L=(b)_I^d$ and hence $a\theta_I^d b$. Thus $ker_Id\subseteq [a]_{\theta_I^d}$. For the converse, let $c\in [a]_{\theta_I^d}$. Again, by Lemma \ref{property}, $(c)_I^d=(a)_I^d=L$ and hence $c\in (c)_I^d$. So $d(c)=d(c \wedge c)\in I$, which implies $c\in ker_Id$. Thus $ker_Id= [a]_{\theta_I^d}$. Since $ker_Id$ is an ideal of $L$, for each $[y]_{\theta_I^d}\in L/{\theta_I^d}$, we get that $a\wedge y\in ker_Id$ and hence $ker_Id=[a]_{\theta_I^d}=[a\wedge y]_{\theta_I^d}\leq [y]_{\theta_I^d}$. Therefore $\bot_{L/\theta_I^d}=ker_I d$.

 ${\rm (ii)}$ The proof is similar to  ${\rm (i)}$.
\end{proof}

As seen in Lemma \ref{property filter}(i), ${\mathcal K}_I^d$ is a filter, whenever ${\mathcal K}_I^d\neq \emptyset$. So in the following lemma we investigate some conditions over which ${\mathcal K}_I^d\neq \emptyset$.
\begin{lemma}\label{nonempty}
${\rm (i)}$ If $\top \in L$, then $\top, d(\top)\in {\mathcal K}_I^d$.

 ${\rm (ii)}$ If $I$ or $ker_Id$ is a prime nontrivial ideal of $L$, then ${\mathcal K}_I^d\neq \emptyset$ and if $ker_Id\neq L$, then $L$ is a disjoint union of $ker_Id$ and ${\mathcal K}_I^d$. Also $\theta_I^d=\{(a,b)\mid \{a,b\}\subseteq ker_Id ~or~ \{a,b\}\subseteq {\mathcal K}_I^d\}$.

 ${\rm (iii)}$ If $L$ is a chain and $I$ a nontrivial ideal of $L$, then ${\mathcal K}_I^d\neq \emptyset$.
\end{lemma}
\begin{proof}
${\rm (i)}$ Is obvious.

 ${\rm (ii)}$ If $ker_Id=L$, then $ker_Id={\mathcal K}_I^d=L$. Let $ker_Id\neq L$ and $b\notin ker_Id$ and $x\in (b)_I^d$. Then $x\wedge d(b)\in I$ and $d(b)\notin I$. If $I$ is prime, $x\in I\subseteq ker_Id$ and if $ker_Id$ is prime, then $x\in ker_Id$ . Thus $b\in {\mathcal K}_I^d$. So $L=ker_Id \cup {\mathcal K}_I^d$ and the first part of the proof will be complete by using Lemma \ref{property filter}.(iii). Now by Proposition \ref{bounded},
$\theta_I^d=\{(a,b)\mid \{a,b\}\subseteq ker_Id ~or~ \{a,b\}\subseteq {\mathcal K}_I^d\}$.

${\rm (iii)}$ It is easy to check that every nontrivial ideal in
a chain is prime. So (ii), completes the proof.
\end{proof}

As a consequence of Lemma \ref{nonempty}(ii), we  conclude that, if $I\subseteq J$, there is no relation between ${\mathcal K}_I^d$ and ${\mathcal K}_J^d$ at all. For example, let $I\subseteq J$ be two prime ideals of $L$ and $d$ be an identity derivation. By Lemmas \ref{nonempty}(ii) and \ref{property}(x), ${\mathcal K}_J^d\subseteq {\mathcal K}_I^d$.

 For  another example, let $L$ have a bottom element $\bot$. Consider $\bot\neq a\in L$, $I=\{\bot\}$, $J=\downarrow a$ and a derivation $d$ defined by $d(x)=a\wedge x$. Clearly $ker_Jd=L$ and, since $d(a)=a\wedge a=a\neq \bot$, $a\notin ker_Id$. So, by Lemma \ref{property filter}(ii), ${\mathcal K}_I^d\subseteq {\mathcal K}_J^d$.

\begin{proposition}\label{greatest}
For a nontrivial ideal $I$ of $L$, the congruence $\theta_I^d$ is the greatest congruence relation having $ker_I d$ as a whole class.
\end{proposition}
\begin{proof}
By Proposition \ref{bounded}, ${\mathcal K}_I^d$ and $ker_I d$ are
whole classes. Let $\theta$ be a lattice congruence on $L$ such
that $ker_I d$ is a whole class and $x\theta y$. The following
cases may  occur:

 $Case \ 1.$ $x,y\in {\mathcal K}_I^d$. Hence $(x)_I^d=ker_Id=(y)_I^d$ and $x\theta_I^d y$.

 $Case\ 2.$ $x,y \notin {\mathcal K}_I^d$. For each $a\in (x)_I^d$, $(x\wedge a)\theta (y\wedge a)$ and $x\wedge a\in ker_Id$. Then $[y\wedge a]_{\theta}=[x\wedge a]_{\theta}=ker_Id$. So $y\wedge a\in ker_Id$ and $a\in (y)_I^d$. Thus $(x)_I^d\subseteq (y)_I^d$ and, by a similar way, $(y)_I^d\subseteq (x)_I^d$, which implies that $x\theta_I^d y$.

 $Case \ 3.$ $x\in {\mathcal K}_I^d$ and $y \notin {\mathcal K}_I^d$ (or similarly $x\in {\mathcal K}_I^d$ and $y \notin {\mathcal K}_I^d$). This case may not  occur. For, consider $b\in (y)_I^d\setminus (x)_I^d$. Then $b\wedge y\in ker_Id$ and $b\wedge x\notin ker_Id$. Also $(b\wedge x)\theta (b\wedge y)$. So $b\wedge x\in ker_Id$, which is impossible. Therefore $\theta\subseteq \theta_I^d$.
\end{proof}

From now on, up to the Lemma \ref{lem 3}, we investigate some conditions over ideals and derivations to get a smallest congruence $\theta_I^d$. The smallest one infer that the quotient lattice $L/\theta_I^d$ has the maximal cardinality.
\begin{proposition}\label{prop1}
For an ideal $I$ and a derivation $d$ on $L$, $\theta^{id}_I\subseteq \theta^d_I$.
\end{proposition}
\begin{proof}
Let $a\theta^{id}_I b$ and $x\in (a)_I^d$. Then $d(x)\in (a)_I^{id}=(b)_I^{id}$. So $d(b\wedge x)=b\wedge id(d(x))\in I$. Thus $x\in (b)_I^d$ which implies $(a)_I^d\subseteq (b)_I^d$ and, by similar way, $(b)_I^d\subseteq (a)_I^d$. So $a\theta_I^d b$.
\end{proof}

The following example shows that $\theta_J^d$ and ${\mathcal K}_J^d$ need not  be larger or smaller with ideal enlargement.

 \begin{example}
${\rm (i)}$ Let $L=\{a,b,c,d\}$ in which $a\prec b\prec c\prec d$, $I=\{a\}$, $J=\{a,b,c\}$ and $f$  an identity derivation on $L$. So $I\subset J$. It is not difficult to check that $(a,b)\in \theta_J^f\setminus \theta_I^f$ and $(b,c)\in \theta_I^f\setminus \theta_J^f$.
Thus $\theta_I^f \subseteq \hspace*{-0.35 cm }/~ \theta_J^f$ and $\theta_J^f\subseteq \hspace*{-0.35 cm }/~\theta_I^f$. Also, by Lemmas \ref{property filter}(ii) and \ref{property}(x), ${\mathcal K}_J^f\subseteq {\mathcal K}_I^f$.

 ${\rm (ii)}$ Let $L=\{a,b,c,d\}$ in which $a$ and $d$ are bottoms
and top element, respectively and $c$ and $d$  have
no relation. Consider $I=\{a\}$ and $J=\{a,b\}$ and $id$ the
identity map. So $I\subset J$. It is not difficult to check that
${\mathcal K}^{id}_I=\{d\}\subset \{c,d\}={\mathcal K}_J^{id}$ and $\theta^{id}_I\subset \theta^{id}_J$.
\end{example}

\begin{lemma}\label{I subset J}
For  ideals $I\subseteq J$ and a derivation $d$ on $L$, if there exists a derivation $d_1$ on $L$ such that $ker_Id_1=J$, then $\theta_I^d\subseteq \theta^d_J$ and the equality  holds if $d_1=d$.
\end{lemma}
\begin{proof}
Let $a \theta_I^d b$ and $x\in (a)^d_J$. Then $d(x\wedge a)\in J=ker_Id_1$, which implies $d_1(x)\wedge d(a)=d_1(d(x\wedge a))\in I$. So $d_1(x)\in (a)_I^d=(b)_I^d$, which implies $d(x\wedge b)\in ker_Id_1=J$. Thus $x\in (b)^d_J$. This gives that $\theta_I^d\subseteq \theta^d_J$.

 Now let $d_1=d$. Consider $a \theta_{ker_Id}^d b$ and $x\in (a)^d_I$. Since $ker_Id$ is an ideal and $d(x)\leq x$, $d(x)\wedge a=d(x\wedge a) \in ker_Id$. So $x\in (a)_{ker_Id}^d=(b)_{ker_Id}^d$ and $d(x\wedge b)\in ker_Id$. Now it is not difficult to show that $x\in (b)_I^d$. Thus $\theta^d_{ker_Id}\subseteq \theta_I^d $.
\end{proof}
Here we have an example in which for ideals $I\subseteq J$ there
is no derivation $d$ on $L$ such that $ker_Id=J$. Suppose that
$L$ is a chain with at least 3 elements and a bottom element
$\bot$. Consider $I=\{\bot\}$ and $J$ a nontrivial ideal of $L$,
which properly contains $I$. Let $d$ be a derivation on $L$ such
that $ker_Id=J$. Consider $\bot\neq x\in J$ and $y\notin J$. So
$x\leq y$ and $x\wedge d(y)=d(x\wedge y)= d(x)=\bot$. Thus
$d(y)=\bot$, because  $L$ is a chain, and hence $y\in J$, which
is impossible.

 \begin{lemma}\label{lem 3}
Let $I$ be an ideal of $L$ and $a\in L$. If $J=(a)_I^d$ and $K$ is an ideal of $L$ such that $I\subseteq K\subseteq (a)_I^d$, then

 ${\rm (i)}$ $(a)_J^d=(a)_I^d=J( a\in {\mathcal K}_J^d)$.

${\rm (ii)}$ $(a)_I^d=(a)_K^d$.

${\rm (iii)}$ $\theta_I^d\subseteq \theta^d_K\subseteq \theta^d_J$.

 ${\rm (iv)}$ $\theta_I^d= \theta^d_J$ whenever $a\in {\mathcal K}_I^d$.
\end{lemma}
\begin{proof}
${\rm (i)}$ By Lemma \ref{property}(x), $(a)_I^d\subseteq (a)_J^d$. Now let $x\in (a)_J^d$. Then $d(x\wedge a)\in J=(a)_I^d$, which implies $a\wedge d(x)=d(d(a\wedge x)\wedge a)\in I$. So $x\in (a)_I^d$.

 ${\rm (ii)}$ This is clear by  (i) and Lemma \ref{property}(x).

 ${\rm (iii)}$ Let $x \theta_I^d y$ and $z\in (x)^d_K$. Then $d(x\wedge z)\in K$, which implies $d(x)\wedge (d(z)\wedge a)=d(d(x\wedge z)\wedge a)\in I$. Since $x \theta_I^d y$, $d(d(y\wedge z)\wedge a)= d(y)\wedge (d(z)\wedge a)\in I$ and hence $d(y\wedge z)\in (a)_I^d=J$. Thus $z\in (y)^d_J=(y)_K^d$.
By a similar way, we can prove $(y)^d_K\subseteq (x)^d_K$. So $(x)^d_K=(y)^d_K$, which deduces that $\theta_I^d\subseteq \theta^d_K$.

 We can prove the inclusion $\theta^d_K\subseteq \theta^d_J$, by a similar way.

 ${\rm (iv)}$ We are done by  Lemma \ref{I subset J}.
\end{proof}

 Note that the converse of Lemma \ref{lem 3}(iii) is not in
generally true. For example, consider $I$ a nontrivial prime ideal
of $L$ and $a\in I$. Then $J=(a)_I^d=L$ and hence for each $x\in L$, $(x)_J^d=L$. So
$\theta_J^d=\nabla$ and, by Lemma \ref{nonempty}, $\theta_I^d=\{(a,b)\mid \{a,b\}\subseteq
ker_Id ~or~ \{a,b\}\subseteq {\mathcal K}_I^d\}$. Thus $\theta_I^d
\neq \theta^d_J$.

In the rest of this section we investigate some relationships between prime ideals and ideals of the form $(x)_I^d$.
First note that, if $I$ is a prime ideal, then so is $ker_Id$.

\begin{lemma}\label{lem 11}
${\rm (i)}$ If $I$ is a prime ideal of $L$, then $ker_Id=L$ or for each $x\notin ker_Id$, $I=ker_Id=(x)_I^d$.

 ${\rm (ii)}$ If $(x)_I^d$ is not a subset of prime ideal $(y)_I^d$, then $x\wedge y\in ker_Id$.

 ${\rm (iii)}$ If $(x)_I^d\neq (y)_I^d$ are prime ideals, then $x\wedge y\in ker_Id$.
\end{lemma}
\begin{proof}
(i) Let $ker_Id\neq L$, $x\notin ker_Id$ and $a\in (x)_I^d$. Thus $a\wedge d(x)\in I$. Since $I$ is prime and $x\notin ker_Id$, $a\in I$.

 (ii) Let $z\in (x)_I^d\setminus (y)_I^d$. Then $x\wedge z\in ker_Id\subseteq (y)_I^d$. Since $(y)_I^d$ is prime, $x\in (y)_I^d$.
\end{proof}

\begin{proposition}
The quotient lattice $L/\theta_I^d=\{ker_Id, [a]_{\theta_I^d},
[b]_{\theta_I^d}\}$ such that for each $x\in [a]_{\theta_I^d}$ and
$y\in [b]_{\theta_I^d}$, $x\wedge y\in ker_Id$
if and only if there exist prime ideals $P_1, P_2$ in $L$ in which $P_1\cup P_2=L$ and $P_1\cap P_2=ker_Id$.
\end{proposition}
\begin{proof}
Let $L/\theta_I^d=\{ker_Id, [a]_{\theta_I^d}, [b]_{\theta_I^d}\}$. First note that, by  Lemma \ref{property}(v), for each $x\in [a]_{\theta_I^d}$, $x\wedge a\notin ker_Id$. The subsets $P_1=[a]_{\theta_I^d}\cup ker_Id$ and $P_2=[b]_{\theta_I^d}\cup ker_Id$ of $L$ are prime ideals. For, let $x,y\in P_1$. In the case where $x\in ker_Id$ or $y\in ker_Id$, by Lemma \ref{property}(i), $x\vee y\in P_1$, else, $(x\vee y)_I^d=(x)_I^d\cap (y)_I^d=(a)_I^d$. Thus $x\vee y\in P_1$. Consiedr $x\in P_1$, $z\in L$ and $z\leq x$. Then $z\wedge b\leq x\wedge b\in ker_Id$. Thus $z\in (b)_I^d$ and hence $z\in P_1$. Now let $x\wedge y\in P_1$ and $y\in [b]_{\theta_I^d}$. So $y\wedge b\notin ker_Id$. If $y\wedge b\in [a]_{\theta_I^d}$, then $y\wedge b=(y\wedge b)\wedge b\in ker_Id$, which is a contradiction. So $y\wedge b\in [b]_{\theta_I^d}$, which implies $x\in P_1$.

For the converse,  consider $V_1=P_1\setminus ker_Id$ and $V_2=P_2\setminus ker_Id$. The subset $V_1$ is a class, for, let $a\in V_1$. We show $V_1=[a]_{\theta_I^d}$. Let $x\in V_1$. For each $y\in (a)_I^d$, $a\wedge y\in ker_I^d\subseteq P_2$ and, since $a\notin P_2$, $y\in P_2$. If $y\in ker_Id$, then $y\in (x)_I^d$, else, $y\in V_2\subseteq P_2$, which implies $x\wedge y \in P_1\cap P_2=ker_Id\subseteq (x)_I^d$. So $y\in (x)_I^d$ and hence $(a)_I^d\subseteq (x)_I^d$. The proof of $(x)_I^d\subseteq (a)_I^d$ is similar. Thus $(a)_I^d= (x)_I^d$, which implies $V_1\subseteq [a]_I^d$. Now let $x\in [a]_I^d$. Then $(x)_I^d=(a)_I^d$ and, since $a\notin ker_Id$, then $x\notin ker_Id$, too. If $x\notin P_1$, then $x\in P_2$ and hence $a\wedge x\in P_1\cap P_2=ker_Id$. Thus $a\in (x)_I^d=(a)_I^d$. By Lemma \ref{property}(iv), $(a)_I^d=L$, which is a contradiction. Thus $x\in P_1$ and hence $x\in V_1$. So $V_1=[a]_{\theta_I^d}$. Similarly, $V_2=[b]_{\theta_I^d}$.
\end{proof}

 \begin{definition}
For a nontrivial ideal $I$ of $L$, an ideal
$P$ is called $I$-minimal, if it is minimal in the set of ideals
containing $I$ and it is called an $I$-minimal prime ideal, if $P$ is a least prime ideal containing $I$.
\end{definition}

 From now on, we consider the set $\Sigma =\{(x)_I^d\mid x\in L\setminus ker_Id\}$. The set $\Sigma$ is a poset under the inclusion relations.

 \begin{theorem}\label{max-min}
Let $I$ be an ideal of $L$ and $a\in I$. The following assertions are equivalent:

 ${\rm (i)}$ $(a)_I^d$ is a maximal element in the $\Sigma$.

 ${\rm (ii)}$ $(a)_I^d$ is a prime ideal.

 ${\rm (iii)}$ $(a)_I^d$ is a $ker_Id$-minimal prime ideal.
\end{theorem}
\begin{proof}
(i)$\Rightarrow$ (ii) Let $x\wedge y\in (a)_I^d$ and $x\notin (a)_I^d$. Since $a\wedge x\leq a$, using Lemma \ref{property}(ii), $(a)_I^d\subseteq (a\wedge x)_I^d$. By the hypothesis, $(a)_I^d= (a\wedge x)_I^d$ or $(a\wedge x)_I^d=L$. If $(a\wedge x)_I^d=L$, then $a\wedge x\in ker_Id$, which is a contradiction. Thus $(a)_I^d= (a\wedge x)_I^d$, which gives that $y\in (a)_I^d$. Now, the proof is complete using Lemma \ref{property}(i).

 (ii)$\Rightarrow$ (iii) Since $(a)_I^d$ is a prime ideal, it is a proper ideal of $L$ and, by Lemma \ref{property}(v), $a\notin ker_Id$. If $ker_Id$ is a prime ideal, we are done, by  Lemma \ref{lem 11}(i). Let $Q$ be a prime ideal of $L$ containing $ker_Id$ such that $Q\subseteq (a)_I^d$ and $x\in (a)_I^d\setminus Q$. Then $x\wedge a\in ker_Id\subseteq Q$. Since $x\notin Q$ and $Q$ is a prime ideal, $a\in Q\subseteq (a)_I^d$. Now, by Lemma \ref{property}(v), $(a)_I^d=L$, which is a contradiction.

 (iii)$\Rightarrow$ (i) Let $(a)_I^d\subseteq (x)_I^d\neq L$. Consider $y\in (x)_I^d\setminus (a)_I^d$. Then $y\wedge x\in ker_Id\subseteq (a)_I^d$, which deduces that $x\in (a)_I^d\subseteq (x)_I^d$. Again, by  Lemma \ref{property}(v), $(x)_I^d=L$, which is a contradiction.
\end{proof}

\begin{lemma}\label{lem 6}
In the following assertions we have, {\rm (i)}$\Rightarrow$ {\rm (ii)}$\Rightarrow$ {\rm (iii)}.

 {\rm (i)} The set $\Sigma$ satisfies the descending chain condition with respect to inclusion.

 {\rm (ii)} $L$ does not have an infinite $M\subseteq L\setminus ker_Id$ such that for each $x,y\in M$, $x\wedge y\in ker_Id$.

 {\rm (iii)} The set $\Sigma$ satisfies the ascending chain condition with respect to inclusion.
\end{lemma}
\begin{proof}
(i)$\Rightarrow$(ii) Let $L$ have an infinite $M\subseteq L\setminus ker_Id$ such that for each $x,y\in M$, $x\wedge y\in ker_Id$ and consider $x_1,x_2\in M$. By Lemma \ref{property}(ii), $(x_1\vee x_2)_I^d\subseteq (x_1)_I^d$ and clearly $x_2\in (x_1)_I^d\setminus (x_1\vee x_2)_I^d$. Thus the following proper descending chain is induced, which is a contradictin:
 $$(x_1)_I^d\supset(x_1\vee x_2)_I^d\supset(x_1\vee x_2\vee x_3)_I^d\supset \cdots$$
(ii)$\Rightarrow$(iii) Let $(a_1)_I^d\subset (a_2)_I^d\subset \cdots$ be a proper chain and $x_j\in (a_j)_I^d\setminus (a_{j-1})_I^d$ for $j=2,3,\cdots$. Consider $y_j=x_j\wedge a_{j-1}\notin ker_Id$. For each $i< j$, since $x_i\in (a_i)_I^d\subseteq (a_{j-1})_I^d$, it is not difficult to show that $y_i\wedge y_j\in ker_Id$. Also, if $y_i=y_j$, then $y_i=y_i\wedge y_j\in ker_Id$, a contradiction. Thus the set $M=\{y_i\mid i=2,3,\cdots\}$ is an infinite set
such that for each $x,y\in M$, $x\wedge y\in ker_Id$, which is a contradiction.
\end{proof}

 We say that the lattice $L$ satisfies the condition $(*)$, if $L$ does not have an infinite $M\subseteq L\setminus ker_Id$ such that for each $x,y\in M$, $x\wedge y\in ker_Id$.

 \begin{lemma} \label{lem 5}
Let $L$ satisfies the condition $(*)$, then
$L$ has only a finite number of distinct $ker_Id$-minimal
prime ideals of the form $(a_i)_I^d(1\leq i\leq n)$. Also
$\bigcap_{_{i=1}}^n (a_i)_I^d=ker_I^d$.

 \end{lemma}
\begin{proof}
By Lemma \ref{lem 6}, $\Sigma$ has maximal elements. Let $(a)_I^d\neq (b)_I^d$ be two maximal element in the set $\Sigma=\{(x)_I^d\mid x\in L\setminus ker_Id\}$. By Lemma \ref{property}(ii), $(a)_I^d\subseteq (a\wedge b)_I^d$ and $(b)_I^d\subseteq (a\wedge b)_I^d$, which implies $ (a\wedge b)_I^d=L$, by the maximality. Using Lemma \ref{property}(v), $a\wedge b\in ker_Id$. So if $\Sigma$ has an infinite number of maximal element , then $L$ has an infinite $M\subseteq L\setminus ker_Id$ such that for each $x,y\in M$, $x\wedge y\in ker_Id$, which is a contradiction. So $\Sigma$ has a finite number of maximal elements. Now Theorem \ref{max-min} complets the first part of the proof.

 Now we show $\bigcap_{_{i=1}}^n (a_i)_I^d=ker_I^d$. Using Lemma \ref{lem 6} and Zorn's lemma, for each $ a\notin ker_Id$ the set $\Sigma=\{(b)_I^d\mid (a)_I^d\subseteq (b)_I^d\}$ has a maximal element. Thus every proper ideal $(a)_I^d$ is contained in a maximal ideal $(a_i)_I^d, 1\leq i\leq n$. Consider $x\in \bigcap_{_{i=1}}^n (a_i)_I^d$. If $(x)_I^d\neq L$, there exists $1\leq i\leq n$ such that $a_i\in (x)_I^d\subseteq (a_i)_I^d$. So $(a_i)_I^d=L$, which is not true. Thus $(x)_I^d=L$ and hence $x\in ker_Id$.
\end{proof}

 \begin{corollary}\label{cor 2}
If $L$ satisfies the condition $(*)$, then every $ker_Id$-minimal prime ideal of $L$ is of the form $(a)_I^d$, for some $a\in L$.
\end{corollary}
\begin{proof}
Let $P$ be a $ker_Id$-minimal prime ideal of $L$. By Lemma \ref{lem 5}, $\bigcap_{_{i=1}}^n (a_i)_I^d=ker_I^d$. Thus $\bigcap_{_{i=1}}^n (a_i)_I^d\subseteq P$ and, since $P$ is a prime ideal, there exists $j\in J$ such that $(a_j)_I^d\subseteq P$, which implies $(a_j)_I^d=P$.
\end{proof}

We close this section by the following important  result, which is an immediate consequence of Corollary \ref{cor 2}.

 \begin{theorem}
If $L$ is a distributive lattice with a bottom element $\bot$ and satisfies the condition $(*)$ for $ker_{\bot}(id)$, then every minimal prime ideal of $L$ is of the form $(a)_{\bot}^{id}$, for some $a\in L$.
\end{theorem}

 A special case of the previous theorem is the case where $L$ is an atomic distributive lattice with a finite number of atoms.

\section{Atomic distributive lattices}
In this section the lattice $L$ considered to be a $ker_Id$-atomic distributive lattice.

\begin{definition}
For an ideal $I$ of $L$, an element $a\in L\setminus I$ is called
$I$-atom, if $\downarrow a\setminus \{a\}=\{x\in L\mid x <
a\}\subseteq I$ and the lattice $L$ is called $I$-atomic if for each $a\in L$ there exists an $I$-atom $a_0$ lower than $a$.
\end{definition}

 From now on consider $A_I^d(L)$, the set of all $ker_Id$-atoms of $L$,
$A_I^d(a)=A_I^d(L)\cap \downarrow a$ and
$A_I^d(a)^c=A_I^d(L)\setminus A_I^d(a)$.
Also consider a subset
$\Gamma_I^d(L)=\bigcup(a_i)_I^d\setminus ker_Id$ in which
$a_i\notin ker_Id$.
As an immediate consequence of Lemma \ref{property}(ii), If $L$ is $ker_Id$-atomic lattice, then $\Gamma_I^d(L)=\bigcup(a_i)_I^d\setminus ker_Id$ in which $a_i\in A_I^d(L)$.

\begin{lemma}\label{lem 13}
${\rm (i)}$ Let $L$ have a top element $\top$. If $\bigvee_{_{j\in J}}
a_j=\top$, for some $ker_Id$-atoms$ (I$-atoms$) a_j$, then
$L\setminus \{\top\}\subseteq \bigcup_{_{j\in J}} (a_j)_I^d$.

 ${\rm (ii)}$ 
$\bigcap_{_{\hspace{-0.5 cm {a_i\in A_I^d(L)}}}} (a_i)_I^d=ker_Id$.
\end{lemma}
\begin{proof}
Let $\top \neq x\in L$. There exists a $ker_Id$-atom ($I$- atom)
$a_j$ such that $a_j \leq {\hspace{-0.3 cm}/} ~~x$. Then
$a_j\wedge x\in \downarrow a_j\setminus \{a_j\}\subseteq ker_Id$,
which means $x\in (a_j)_I^d$.

(ii) By Lemma \ref{property}, $ker_Id\subseteq \bigcap_{_{\hspace{-0.5 cm {a_i\in A_I^d(L)}}}} (a_i)_I^d$. For the converse let $x\in \bigcap_{_{\hspace{-0.5 cm {a_i\in A_I^d(L)}}}} (a_i)_I^d\setminus ker_Id$. Then there exists a $ker_Id$-atom $a\leq x$ in such a way that $x\wedge a=a\notin ker_Id$. So $x\notin (a)_I^d$, which is impossible.
\end{proof}

In part (i) of the following lemma we get  another definition for the congruence $\theta_I^d$.

 \begin{lemma}\label{lem 9}
Let $L$ be a $ker_Id$-atomic distributive lattice. If $a,b\in L$,
then

 ${\rm (i)}$ $a\theta_I^d b$ if and only if $A_I^d(a)=A_I^d(b)$.

 ${\rm (ii)}$ $a\wedge b\in ker_Id$ if and only if $A_I^d(a)\cap A_I^d(b)=\emptyset$.

 ${\rm (iii)}$ For an element $a\in L$, if $A_I^d(L)=A_I^d(a)$, then $a\in {\mathcal K}_I^d$.

${\rm (iv)}$ If $x=\bigvee_{_{\hspace{-0.5 cm {a_i\in A_I^d(L)}}}} a_i$, then $x\in {\mathcal K}_I^d$.
\end{lemma}
\begin{proof}
It is easy to check the parts (i) and (ii).

(iii) Since $A_I^d(L)=A_I^d(a)$, for each $a_i\in A_I^d(L)$, $a_i\leq a$. By
 Lemmas \ref{property}(ii) and \ref{lem 13}(ii), $(a)_I^d
\subseteq \bigcap_{_{\hspace{-0.5 cm {a_i\in A_I^d(L)}}}} (a_i)_I^d=ker_Id\subseteq (a)_I^d$. Then $a\in {\mathcal K}_I^d$.

 (iv) Straightforward, by  (iii).
\end{proof}

\begin{lemma}\label{lem 12}
If $a\in A_I^d(L)$, then $(a)_I^d$ is a maximal element in the set $\Sigma$.
\end{lemma}
\begin{proof}
Let $(a)_I^d\subseteq (b)_I^d$. If $a\wedge b\in ker_Id$, then $b\in (a)_I^d\subseteq (b)_I^d$ and, by Lemma \ref{property}(v), $(b)_I^d=L$. In the case where $a\wedge b\notin ker_Id$, since $a$ is $ker_Id$-atom, $a=a\wedge b$ and hence $a\leq b$. By Lemma \ref{property}(ii), $(b)_I^d\subseteq (a)_I^d$, which implies $(b)_I^d= (a)_I^d$. Thus $(a)_I^d$ is a maximal element in the set $\Sigma$.
\end{proof}

For an immediate consequence of Theorem \ref{max-min}, Lemmas \ref{property}(ii) and \ref{lem 12},
we have the following theorem.

 \begin{theorem}\label{th 4}
For an element $a\in L$, $(a)_I^d$ is a $ker_Id$-minimal prime ideal if and only if there exists a $ker_Id$-atom $a_0$ such that $A_I^d(a)=\{a_0\}$.
\end{theorem}

 \begin{lemma}\label{lem 7}
Let $L$ satisfy the condition $(*)$. Then

 ${\rm (i)}$ Every $ker_Id$-minimal prime ideal of $L$ is of the form $(a)_I^d$, for some $a\in A_I^d(L)$.

${\rm (ii)}$ If $L$ is atomic, then every minimal prime ideal of
$L$ is of the form $(a)_I^d$, for some atom $a$.

 \end{lemma}
\begin{proof}
(i) Let $P$ be a $ker_Id$-minimal prime ideal of $L$. Using Corollary \ref{cor 2}, $P=(a_j)_I^d$ for some $a_j\in L$. Now, by Theorem \ref{th 4}, there exists $a\in A_I^d(L)$ such that $A_I^d(a_j)=\{a\}$. Thus, by Lemma \ref{lem 9}(i), $P=(a_j)_I^d=(a)_I^d$.
\end{proof}

\begin{theorem}\label{th 1}
Let $L$ satisfy the condition $(*)$, then
$L$ has only a finite number of distinct $ker_Id$-minimal
prime ideals $P_i(1\leq i\leq n)$. Furthermore,
$\bigcap_{_{i=1}}^n P_i=ker_Id$, $\bigcap_{_{i\neq j}}P_i\neq
ker_Id$ for all $1\leq j\leq n$ and $L\setminus
\bigcup_{_{i=1}}^n P_i={\mathcal K}_I^d$.
\end{theorem}
\begin{proof}
By  Lemmas \ref{lem 5} and \ref{lem 7}, $L$ has only a
finite number of distinct $ker_Id$-minimal prime ideals $P_i(1\leq
i\leq n)$, in which $\bigcap_{_{i=1}}^n P_i=ker_Id$. Let for the
fixed index $j$, $\bigcap_{_{i\neq j}}P_i= ker_Id$. By Lemma
\ref{lem 7}, each $P_i$ is of the form $(a_i)_I^d$. Consider
$x_i\in (a_i)_I^d\setminus (a_j)_I^d$. Then $\bigwedge_{_{i\neq
j}}x_i\in \bigcap_{_{i\neq j}}P_i= ker_Id\subseteq (a_j)_I^d$.
Since $(a_j)_I^d$ is a prime ideal, there is an $i\neq j$ such
that $x_i\in (a_j)_I^d$, which is a contradiction. Therefore
$\bigcap_{_{i\neq j}}P_i\neq ker_Id$ for all $1\leq j\leq n$. Now
we show $L\setminus \bigcup P_i={\mathcal K}_I^d$. Let $x\in
L\setminus \bigcup P_i$ and $y\notin ker_Id$. If $x\wedge y\in
ker_Id$, then using Theorem \ref{max-min} and Lemma \ref{lem 6}, there exists $1\leq i\leq n$ such that $x\in (y)_I^d\subseteq (a_i)_I^d=P_i$, which is a contradiction. Thus $y\notin (x)_I^d$
and hence $(x)_I^d=ker_Id$. So $x\in {\mathcal K}_I^d$. Now
consider $x\in {\mathcal K}_I^d$. If there exists $1\leq i\leq n$
in which $x\in (a_i)_I^d$, then $a_i\in (x)_I^d=ker_Id\subseteq
(a_i)_I^d$, which is a contradiction. Thus $x\in L\setminus
\bigcup P_i$ and hence $L\setminus \bigcup P_i={\mathcal
K}_I^d$.
\end{proof}

\begin{corollary}
If $L$ has a bottom element $\bot$ and does not have an infinite $M\subseteq L\setminus \{\bot\}$ such that for each $x,y\in M$, $x\wedge y=\bot$, then $L$ has only a finite number of minimal prime ideals.

 \end{corollary}

\begin{theorem}
The following assertions are equivalent:

 ${\rm (i)}$ $L$ satisfies the condition $(*)$.

 ${\rm (ii)}$ There exists a finite number of minimal $ker_Id$-prime ideals $P_i(1\leq i\leq n)$ such that $\bigcap_{_{i=1}}^nP_i=ker_Id$.

 \end{theorem}
\begin{proof}
(i)$\Rightarrow$(ii) We are done, by  Theorem \ref{th 1}.

 (ii)$\Rightarrow$(i) Let $M\subseteq L\setminus ker_Id$ such that for each $x,y\in M$, $x\wedge y\in ker_Id$ and $\mid M \mid \geq n$. By Pigeonhole principle, there exist $x,y\in M$ and $P_i$ such that $x,y\in P_i^c$, which is a contradiction, because, $P_i$ is prime and $x\wedge y\in ker_Id\subseteq P_i$.
 \end{proof}

\begin{proposition}
Let $x, y\in \Gamma_I^d(L)$. Then only one of the following cases
may  occur:

 ${\rm (i)}$ $x\wedge y\in ker_Id$.

 ${\rm (ii)}$ There exists an element $z\in L$ such that $ker_Id\neq (z)_I^d\neq L$ and $x,y\in (z)_I^d$.

 ${\rm (iii)}$ There exist $z_1, z_2\in L$ such that $ker_Id\neq (z_1)_I^d\neq L$, $ker_Id\neq (z_2)_I^d\neq L$ and $x\wedge z_1$, $z_1\wedge z_2$, $z_2\wedge y \in ker_Id$.
\end{proposition}
\begin{proof}
Let $x\wedge y\notin ker_Id$. There exist $a_1,a_2\notin ker_Id$
such that $x\in (a_1)_I^d$ and $y\in (a_2)_I^d$. Since
$x,a_1\notin ker_Id$, $ker_Id\neq (a_1)_I^d\neq L$. Also for
$(a_2)_I^d$. Now two cases may  occur:

 (i) If $a_1\wedge a_2\notin ker_Id$, then $x,y\in (a_1\wedge a_2)_I^d$.

 ${\rm (ii)}$ If $a_1\wedge a_2\in ker_Id$, then $x\wedge a_1$, $a_1\wedge a_2$, $a_2\wedge y \in ker_Id$.
\end{proof}

For a subset $A$ of $L$, the set of all upper bounds of the elements of $A$ is denoted by $\uparrow A$.
\begin{theorem}
Let $L$ be a $ker_Id$-atomic distributive lattice. Then for each
$a\in \Gamma_I^d(L)$, $(a)_I^d=\uparrow A_I^d(a)^c\setminus \uparrow A_I^d(a)$.
\end{theorem}
\begin{proof}
Let $x\in (a)_I^d$. If $x\in \uparrow A_I^d(a)$, then there exists $c\in A_I^d(a)$ such that $c\leq x$. Hence $c\leq x\wedge a\in ker_Id$, which is impossible. So $x\in \uparrow A_I^d(a)^c\setminus \uparrow A_I^d(a)$. For the converse, assume that $x\in \uparrow A_I^d(a)^c\setminus \uparrow A_I^d(a)$. If $x\notin (a)_I^d$, then $a\wedge x\notin ker_Id$ and so $x\in \uparrow A_I^d(a\wedge x)\subseteq \uparrow A_I^d(a)$, which is a contradiction.
\end{proof}

Consider $C_I^d(L)=\{B\subseteq L\setminus ker_Id\mid \forall x,y\in B, x\wedge y\in ker_Id\}$. It is easy to check that $A_I^d(L)\in C_I^d(L)$.
\begin{theorem}
If $L$ is a $ker_Id$-atomic lattice, then for each $B\in C_I^d(L)$, $\mid B\mid \leq \mid A_I^d(L)\mid$.
\end{theorem}
\begin{proof}
Let $B\in C_I^d(L)$ and $x,y\in B$. By Lemma \ref{lem 9}(ii), $A_I^d(x)\cap A_I^d(y)=\emptyset$ and $A_I^d(x)\neq \emptyset$ in which $A_I^d(y)\neq \emptyset$. By the axiom of choice, for each $b\in B$, choose and fix $a_b\in A_I^d(b)\neq \emptyset$. So the map $f:B\rightarrow A_I^d(L)$, defined by $f(b)=a_b$, is a one to one map. Hence $\mid B\mid \leq \mid A_I^d(L) \mid$.
\end{proof}

\maketitle
\section{when a quotient lattice is a Boolean algebra}

 In this section some necessary and sufficient conditions are derived for the
quotient algebra $L/\theta$ to become a Boolean algebra.

 For a distributive lattice $L$ and a lattice congruence $\theta$ on $L$, we mean the set $[x]_{\theta}=\{y\in L\mid x\theta y\}$ a congruence class of $x$. The set of all congruence classes of $L$ with respect to $\theta$, is denoted by $L/\theta$. It can be easily observed that $L/\theta$ is a distributive lattice with the following operations $[x]_{\theta}\wedge [y]_{\theta}=[x\wedge y]_{\theta}$ and $[x]_{\theta}\vee [y]_{\theta}=[x\vee y]_{\theta}$.

 \begin{theorem}\label{boolean}
Let $L$ be a distributive lattice and $\theta$ a lattice congruence on $L$. The distributive lattice $L/\theta$ is a Boolean algebra if and only if the following conditions hold:

 ${\rm (i)}$ There exists $a_0, b_0\in L$ such that for each $x\in L$, $[a_0]_{\theta}\leq [x]_{\theta}\leq [b_0]_{\theta}$, which means that $\bot_{L/\theta}=[a_0]_{\theta}$ and $\top_{L/\theta}=[b_0]_{\theta}$.

 ${\rm (ii)}$ For each $x\in L$ there exists $y\in L$ such that
$(x\wedge y)\theta a_0$ and $(x\vee y)\theta b_0$.
\end{theorem}
\begin{proof}
Let $L/\theta$ be a Boolean algebra. Thus $L/\theta$ has both a least and a greatest element, which means there exist $a_0$ and $b_0$ in $L$ such that $\bot_{L/\theta}=[a_0]_{\theta}$ and $\top_{L/\theta}=[b_0]_{\theta}$. So the statement (i) holds. Now let $x\in L$. Since $L/\theta$ is a Boolean algebra, there exists $[y]_{\theta}\in L/\theta$ such that $[x\wedge y]_{\theta}=[x]_{\theta}\wedge [y]_{\theta}=[a_0]_{\theta}$ and $[x\vee y]_{\theta}=[x]_{\theta}\vee [y]_{\theta}=[b_0]_{\theta}$ and hence $(x\wedge y)\theta a_0$ and $(x\vee y)\theta b_0$.

 The proof of the converse is obvious.
\end{proof}

 Combining Theorem \ref{boolean}, Proposition \ref{bounded}, and
Proposition \ref{greatest}, one can obtain the following theorem,
which is one of the main results in this article. Also see \cite[Th.2.8]{sambasiva}, for the case where $I=\{\bot\}$.

 \begin{theorem}\label{boolean1}
Let $I$ be a nontrivial ideal of $L$. Then $L/{\theta_I^d}$ is a Boolean algebra if and only if for each $x\in L$, there exists $y\in (x)_I^d$ such that $x\vee y\in {\mathcal K}_I^d$.
\end{theorem}

 \begin{corollary}\label{lem 8}
Let $L/\theta_I^d$ be a Boolean algebra. Then $[x]_{\theta_I^d}^{-1}=[y]_{\theta_I^d}$ if and only if $x\wedge y\in ker_Id$ and $x\vee y\in {\mathcal K}_I^d$ .
\end{corollary}

 \begin{proposition}\label{prime}
${\rm (i)}$ If $I$ or $ker_Id$ is a prime ideal of $L$, then $L/\theta_I^d$ is a Boolean algebra.

 ${\rm (ii)}$ If each $(x)_I^d$ has a maximum element, then $L/\theta_I^d$ is a Boolean algebra.
\end{proposition}
\begin{proof}
${\rm (i)}$ If $ker_Id=L$, then $ker_Id={\mathcal K}_I^d=L$. Thus
$\theta_I^d=\nabla$ and $L/\theta_I^d$ is a singleton set. Let
$ker_Id\neq L$ and $x\in L$. By Lemma \ref{nonempty}(ii),
${\mathcal K}_I^d\neq \emptyset$ and $L$ is a disjoint union of
$ker_Id$ and ${\mathcal K}_I^d$. Consider $a\in I$ and $b\in
{\mathcal K}_I^d$. If $x\in {\mathcal K_I^d}$, then $x\vee b\in
{\mathcal K}_I^d$ and $x\wedge b\in ker_Id$ and if $x\in
ker_Id$, then $x\wedge a\in ker_Id$ and
$x\vee a\in {\mathcal K}_I^d$. So we are done, by Theorem \ref{boolean1}.

 ${\rm (ii)}$ If $ker_Id=L$, then $ker_Id={\mathcal K}_I^d=L$. Thus $\theta_I^d=\nabla$ and $L/\theta_I^d$ is a singleton set. If ${\mathcal K}_I^d=L$, then for each $a,b\in L$, $(a)_I^d=ker_Id=(b)_I^d$. Thus $\theta_I^d=\nabla$ and $L/\theta_I^d$ is a singleton set. Let $ker_Id$ and ${\mathcal K}_I^d$ be nontrivial and $x\in L$. 
Consider $a_0\in ker_Id$ and $b_0\in {\mathcal K}_I^d$. If $x\in
ker_Id$, then $x\wedge b_0\in ker_Id$ and $x\vee b_0\in {\mathcal
K}_I^d$. If $x\in {\mathcal K}_I^d$, then $x\wedge a_0\in ker_Id$
and $x\vee a_0\in {\mathcal K}_I^d$. Now, let $x\notin ker_Id \cup
{\mathcal K}_I^d$ and $y$ be the maximum element of $(x)_I^d$.
Then $x\wedge y\in ker_Id$. We show that $x\vee y\in {\mathcal
K}_I^d$. Let $z\in (x\vee y)_I^d=(x)_I^d\cap (y)_I^d$. Since $y$
is a maximum element of $(x)_I^d$, $z=(x\wedge z)\vee z=(x\wedge
z)\vee (y\wedge z)= (x\vee y)\wedge z\in ker_Id$. Thus $(x\vee
y)_I^d\subseteq ker_I^d$ and, by Lemma \ref{property}(iv), $x\vee
y\in {\mathcal K}_I^d$.  So, Theorem \ref{boolean1} completes the
proof.
\end{proof}

One of the important especial case of Proposition \ref{prime}(i) is the case where $L$ is a chain.

\begin{lemma}\label{}
If $L$ is a Boolean algebra with a bottom element $\bot$, then $\theta_{\bot}^{id}=\Delta=\{(a,a)\mid a\in L\}$.
\end{lemma}
\begin{proof}
It is clear that $ker_Id=\{\bot\}$ and $(a)_I^d=\downarrow a'$, where $a'$ is the complement of $a$. If $a\theta_{\bot}^{id} b$, then $\downarrow a'=\downarrow b'$ and hence $a'=b'$. Thus $a=b$, which implies $\theta_{\bot}^{id}=\Delta$.
\end{proof}

 By Corollary \ref{der-hom}, every derivation is a lattice homomorphism. So for a derivation $d$, $ker(d)=\{(a,b)\mid d(a)=d(b)\}$ is a lattice congruence on $L$.

It is not difficult to show that for a nontrivial ideal $I$ and a derivation $d$, $ker(d)\subseteq \theta_I^d$, but the converse is not  generally true. For example, consider $I\neq \bot$ and $d=id$. Then $ker(d)=\Delta$ and for each $x, y\in I$, $(x)_I^d=(y)_I^d=L$. So $x\theta_I^d y$. In the case where $I=\{\bot\}$, using Lemma \ref{property}(v), $\theta_{\bot}^d=\nabla$ deduces that $ker(d)=\nabla$. The following lemma show that in the case where, $L$ is a Boolean algebra with a bottom element $\bot$, then $\theta_{\bot}^d=ker(d)$ at all.

 \begin{lemma}
Let $L$ be a Boolean algebra with a bottom element $\bot$ and $d$ a derivation on $L$. Then $ker(d)=\theta_{\bot}^d$.
\end{lemma}
\begin{proof}
Let $x\theta_{\bot}^d y$. Since $L$ is a Boolean algebra, $y$ has a complement element $y'$ in which $y'\in (y)_{\bot}^d=(x)_{\bot}^d$. Thus $d(x)\wedge d(y')=\bot$ and hence $d(y)\vee d(x)=d(y)\vee (d(x)\wedge d(y'))=d(y)\vee \bot=d(y)$. So $d(x)\leq d(y)$ and, by a similar way, $d(y)\leq d(y)$. Therefore $(x, y)\in ker(d)$.
\end{proof}

The following theorem is  another version of \cite[Th. 3.4]{sambasiva}.
\begin{theorem}\label{singleton}
Let $I$ be an ideal of $L$ and $d$ a derivation on $L$. Then the following are equivalent:

 ${\rm (i)}$ $\theta_I^d=\nabla$.

${\rm (ii)}$ $ker_Id=L$

${\rm (iii)}$ For each $x\in L$, $I\cap [x]_{ker(d)}$ is a singleton set.
\end{theorem}
\begin{proof}
${\rm (i) \Rightarrow(ii)}$ Let $x\in L$ and $a\in ker_Id$. Since $\theta_I^d=\nabla$, $x\theta_I^d a$ and, by Lemma \ref{property}(v), $x\in ker_Id$. So $ker_Id=L$.

 ${\rm (ii) \Rightarrow(iii)}$ From the part one of the proof of \cite[Th. 3.4]{sambasiva}.

 ${\rm (iii) \Rightarrow(i)}$ Let $x,y\in L$. Consider $I\cap [x]_{ker(d)}=\{x_0\}$ and $I\cap [y]_{ker(d)}=\{y_0\}$. By Lemma \ref{property-derivation}(ii), $d(x)=d(x_0)\leq x_0$ and, since $I$ is an ideal, $d(x)\in I$. By Lemma \ref{property-derivation}(ii), $d(x)\in I\cap [x]_{ker(d)}$, which implies $d(x)=x_0$. Similarly $d(y)=y_0$. Using Lemma \ref{property filter}(iv) and Proposition \ref{bounded}(i), $x\theta_I^d x_0\theta_I^d y_0\theta_I^d y$. Thus $\theta_I^d=\nabla$.
\end{proof}

\begin{proposition}\label{lem 4}
The Boolean algebra $L/\theta_I^d={\bf 2}$ if and only if $ker_Id$ is a prime ideal of $L$.
\end{proposition}
\begin{proof}
Let $L/\theta_I^d={\bf 2}$, $x\wedge y\in ker_Id$ and $x,y\in L\setminus ker_Id$. Since $L/\theta_I^d={\bf 2}$, by Proposition \ref{bounded}(i), $x\theta_I^d y$. So $x\in (y)_I^d=(x)_I^d=ker_Id$.
This implies $x\in (x)_I^d$, which contradicts Lemma \ref{property}(v).

The converse one gets using Lemma \ref{nonempty}.
\end{proof}

 Here we have an example for $L/\theta_I^d= {\bf 2}$, but $I$ is not prime. Consider the four element lattice $\{\bot, a, b, \top\}$, in which $\bot$ and $\top$ are bottom and top element, respectively and $a, b$ are not comparable. The map $d:L\rightarrow L$ defined by
$d(x) = \left\{
\begin{array}{ll}
\bot,& {\rm if}\ x=\bot ,b\\
a,& {\rm if}\ x=a,\top\\
\end{array} \right.$
is a derivation. It is clear that $ker_Id=\{\bot, b\}$ and ${\mathcal K}_I^d=\{a, \top\}$. By Proposition \ref{bounded}, $L/\theta_I^d={\bf 2}$, but $I=\{\bot\}$ is not a prime ideal.

Consider the set $\Sigma=\{(x)_I^d\mid x\in L\}$. By an order defined as follow, the set $\Sigma$ is a poset. For each $x,y\in L$, $(x)_I^d\leq (y)_I^d$ if and only if $(y)_I^d \subseteq (x)_I^d$.

 Also, with the following operations, $\Sigma$ is a bounded distributive lattice. For each $x,y\in L$, $(x)_I^d\vee (y)_I^d=(x \vee y)_I^d$ and $(x)_I^d \wedge (y)_I^d=(x \wedge y)_I^d$. The bottom and the top elements in the lattice $\Sigma$ are  of the form, $\bot_{\Sigma}=(x)_I^d=L$ for each $x\in ker_Id$ and $\top_{\Sigma}=(x)_I^d=ker_Id$ for each $x\in {\mathcal K}_I^d$. The map $f:L\rightarrow \Sigma$ defined by $f(x)=(x)_I^d$ is a lattice epimorphism, in which $kerf=\theta_I^d$. Thus, by the Isomorphism Theorem, $L/\theta_I^d\cong \Sigma$.

\begin{lemma}\label{max}
If the quotient lattice $L/\theta_I^d$ is a Boolean algebra then for each $x\in L$, the set $\{(z)_I^d\mid z\in (x)_I^d\}$ has a maximum element.
\end{lemma}
\begin{proof}
Let $L/\theta_I^d$ be a Boolean algebra and $x\in L$. By Theorem \ref{boolean1}, there exists $y\in L$ such that $x\wedge y\in ker_I d$ and $x\vee y\in {\mathcal K}_I^d$. Consider $z\in (x)_I^d$. Since $x\wedge z\in ker_Id$, applying Proposition \ref{greatest}, $(x\wedge y)\theta_I^d (x\wedge
z)$. Thus $y\theta_I^d [y\vee (x\wedge z)]= [(x\vee y)\wedge (y\vee z)]\theta_I^d (y\vee z)$. So $(y)_I^d=(y\vee z)_I^d=(y)_I^d\cap (z)_I^d$ and hence $(y)_I^d\subseteq (z)_I^d$, which deduces that $(z)_I^d\leq (y)_I^d$.
\end{proof}

\begin{theorem}\label{th 3}
Let $L$ be a $ker_Id$-atomic distributive lattice. The lattice $L/\theta_I^d$ is a Boolean algebra if and only if for each $x\in L$, there exists $y\in L$ such that $A_I^d(x)$ and $A_I^d(y)$ are a partition of $A_I^d(L)$ and $[y]_{\theta_I^d}$ is a complement of $[x]_{\theta_I^d}$ in $L/\theta_I^d$.
\end{theorem}
\begin{proof}
$(\Leftarrow)$ It is clear that $x\wedge y\in ker_Id$ and, by
Lemma \ref{lem 9}, $x\vee y\in {\mathcal K}_I^d$. Hence, Theorem
\ref{boolean1} completes the proof.

$(\Rightarrow)$ Consider $x\in L$. Since $L/\theta_I^d$ is a Boolean algebra, by Theorem \ref{boolean1}, there exists $y\in L$ such that $x\wedge y\in ker_Id$ and $x\vee y\in {\mathcal K}_I^d$. Clearly $A_I^d(x)\cap A_I^d(y)=\emptyset$. Let $a\in A_I^d(L)\setminus (A_I^d(x)\cup A_I^d(y))$. Using Lemma \ref{property filter}(i), $(x\vee a)\vee y\in {\mathcal K}_I^d$. Also $(x\vee a)\wedge y\in ker_Id$. So, by Corollary \ref{lem 8}, $[y]_{\theta_I^d}$ has two different complements $[x]_{\theta_I^d}$ and $[x\vee a]_{\theta_I^d}$, which is a contradiction, because $a\in (x)_I^d$ and $a\notin (x\vee a)_I^d$.
\end{proof}

 \begin{theorem}\label{th 2}
If $L/\theta_I^d$ is a Boolean algebra, then the congruence $\theta_I^d$ is the only congruence relation having $ker_I d$ as a whole class.
\end{theorem}
\begin{proof}
Let $\theta$ be a lattice congruence on $L$ such that $ker_I d$ is a whole class. By Proposition \ref{greatest}, $\theta\subseteq \theta_I^d$. For the converse, let $x\theta_I^d y$. Then there exists $z\in L$ such that $[x]_{\theta_I^d}^{-1}=[y]_{\theta_I^d}^{-1}=[z]_{\theta_I^d}$. By  Proposition \ref{bounded}, $[x\wedge z]_{\theta_I^d}=[x]_{\theta_I^d}\wedge [x]_{\theta_I^d}= \bot_{L/\theta_I^d}=ker_Id$. Thus $x\wedge z\in ker_Id$ and also $y\wedge z\in ker_Id$, which implies $(x\wedge z)\theta (y\wedge z)$. By a similar way, $(x\vee z)\theta (y\vee z)$. Now we have $x=x\vee (x\wedge z)\theta [x\vee (y\wedge z)]\theta [(x\vee y)\wedge (x\vee z)] \theta [(x\vee y)\wedge (y\vee z)]= [y\vee (x\wedge z)]\theta [y\vee (y\wedge z)]=y$. Thus $\theta_I^d\subseteq \theta$ and hence $\theta_I^d= \theta$.
\end{proof}

 \begin{corollary}
For a congruence $\theta$,
if $L/\theta_I^d$ and $L/\theta$ are Boolean algebras such that the congruence $\theta$ having $ker_I d$ as a whole class, then $\theta_I^d=\theta$.

 \end{corollary}

\begin{corollary}
If $L$ is a distributive lattice with a least element $\bot$, $ker_Id=\{\bot\}$ and $L/\theta_I^d$ is a Boolean algebra, then $\theta_I^d=\Delta$.
\end{corollary}

\newpage

$$ Conclusion$$
In this final section, for an ideal $I$, we conclude that
the lattice congruence $\theta_I^{id}$ is the smallest and so the best congruence in some classes of congruences as the following cases, because, if $L/\theta_I^{id}$ is a Boolean algebra, then $L/\theta_I^{id}$ has the maximum cardinality in the set of all $L/\theta_I^{d}$.

 {\rm (i)} Consider an ideal $I$ and a derivation $d$ on $L$. By Proposition \ref{prop1}, $\theta^{id}_I\subseteq \theta^d_I$. Thus the map $\pi:L/\theta_I^{id}\rightarrow L/\theta_I^d$ defined by $\pi([a]_{\theta_I^{id}})=[a]_{\theta_I^d}$ is a lattice homomorphism. Using First Isomorphism Theorem, if $L/\theta_I^{id}$ is a Boolean algebra, then so is $L/\theta_I^d$. Thus the lattice congruence $\theta_I^{id}$ is the best congruence in the set $\{\theta_I^{d}\mid$ d~is~a~derivation$\}$.

{\rm (ii)} Combining Theorem \ref{th 2} and Proposition \ref{prop1}, it is concluded that $\theta_I^{id}$ is the smallest congruence in the set of all congruences having $ker_Id$ as a whole class.

 {\rm (iii)} Using Lemma \ref{I subset J}, $\theta_I^{id}$ is the smallest congruence in the set $\{\theta_J^d\}$ in which there exists a derivation $d$ on $L$ such that $ker_Id=J$.

{\rm (iv)} Using Lemma \ref{lem 3}, $\theta_I^{id}$ is the smallest congruence in the set $\{\theta_J^d\}$ in which $J=(a)_I^d$, for all $a\in L$.

 {\rm (v)} Using Theorem \ref{th 2}, $\theta_I^{id}$ is the smallest congruence in the set of all congruences having $I$ as a whole class.

 {\rm (vi)} In the case where $L$ is a $ker_Id$-atomic distributive lattice such that for each $x\in L$, there exists $y\in L$ such that $A_I^d(x)$ and $A_I^d(y)$ are a partition of $A_I^d(L)$, then $\theta_I^{id}$ is the smallest congruence in which $L/\theta_I^{d}$ is a Boolean algebra.

 There is still an open question concerning $\theta_I^{d}$:

 Is there a necessary and sufficient condition on an ideal $I$ such that $\theta_I^{d}$ is the smallest congruence in which $L/\theta_I^{d}$ is a Boolean algebra at all.


 \end{document}